\documentclass[11pt]{amsart}
\usepackage[margin=1in]{geometry}
\calclayout
\usepackage[utf8]{inputenc}
\usepackage{amsmath, amsthm, amssymb, mathtools, enumitem, breqn, footnote, booktabs}
\usepackage[symbol]{footmisc}
\usepackage[colorlinks=true, citecolor=blue]{hyperref}

\newtheorem{fact}{Fact}[section]
\newtheorem{define}{Definition}[section]
\newtheorem{prop}{Proposition}[section]
\newtheorem{theorem}{Theorem}[section]
\newtheorem{lemma}{Lemma}[section]
\newtheorem{cor}{Corollary}[section]
\newtheorem{conj}{Conjecture}[section]
\newtheorem{remark}{Remark}[section]

\newcommand{\norm}[1]{\lVert #1 \rVert}
\newcommand{\Norm}[1]{\left\lVert #1 \right\rVert}

\newcommand{\floor}[1]{\lfloor #1 \rfloor}
\newcommand{\Floor}[1]{\left\lfloor #1 \right\rfloor}
\DeclareMathOperator{\ML}{ML}

\title{Amending the Lonely Runner Spectrum Conjecture}

\author{Ho Tin Fan}
\address[Ho Tin Fan]{Massachusetts Institute of Technology}
\email{codetiger927@gmail.com}

\author{Alec Sun}
\address[Alec Sun]{University of Chicago}
\email{sundogx@gmail.com}

\begin{document}

\begin{abstract}
    Let $||x||$ be the absolute distance from $x$ to the nearest integer. For a set of distinct positive integral speeds $v_1, \ldots, v_n$, we define its maximum loneliness, also known as the gap $\delta$, to be
    $$ML(v_1,\ldots,v_n) = \max_{t \in \mathbb{R}}\min_{1 \leq i \leq n} || tv_i||.$$
    The Loneliness Spectrum Conjecture, recently proposed by Kravitz (2021), asserts that
    $$\exists s \in \mathbb{N}, \text{ML}(v_1,\ldots,v_n) = \frac{s} {sn + 1} \text{ or } \text{ML}(v_1,\ldots,v_n) \geq \frac{1}{n}.
    $$
    We disprove the Loneliness Spectrum Conjecture for $n = 4$ with an infinite family of counterexamples and propose an alternative conjecture. We confirm the amended conjecture for $n = 4$ whenever there exists a pair of speeds with a common factor of at least $3$ and also prove some related results.
\end{abstract}

\maketitle

\section{Introduction}

A long-standing problem in number theory is the Lonely Runner Conjecture, first introduced by Wills in 1965 \cite{willsthesis, wills} and now in its 60th anniversary \cite{perarnau}. The conjecture focuses on the following scenario: Let $n$ runners start on the same point and each run at a pairwise-distinct constant velocity on a circular unit-length race track, namely $\mathbb{R}/\mathbb{Z}$. At a particular moment in time, a runner is considered ``lonely'' if he or she is at least $\frac{1}{n}$ units of distance away from every other runner. The Lonely Runner Conjecture asserts that all runners will sooner or later become lonely regardless of the speeds, possibly at different times. 

\par Another popular reformulation of the problem arises when we focus on one runner's frame of reference. Each of the other runners' velocities is subtracted by the chosen runner's velocity, while the chosen runner remains fixed at the start point. We define the loneliness of $n$ moving runners with nonzero velocities $v_1,\ldots,v_n \in \mathbb{R}$ as the maximum achievable distance $L$ such that there exists a time when all $n$ runners are $L$ away from the start point. Let $\norm{x}$ denote the absolute distance from $x$ to the nearest integer, then the loneliness value, also known as the gap $\delta$ in other papers, can be more formally written as
$$\ML(v_1,\ldots,v_n) = \delta(v_1,\ldots,v_n) = \max_{t \in \mathbb{R}}\min_{1 \leq i \leq n} \norm{tv_i}.$$
\par We further note that relative to one runner, the sign of other runners' velocities is irrelevant, so we can further restrict our discussion to only positive speeds. The Lonely Runner Conjecture for $n + 1$ runners thus asserts the following condition regarding the loneliness value.
\begin{conj} \label{lrc}
    [Wills, Cusick] For any set of $n$ nonzero positive speeds $v_1,\ldots,v_n \in \mathbb{R}^+$,
    $$\ML(v_1,\ldots,v_n) \geq \frac{1}{n + 1}.$$
\end{conj}
\par So far Conjecture~\ref{lrc} has been proven for up to $7$ runners \cite{barajas}. More recently, computer-assisted proofs have established the conjecture for $8$ runners \cite{rosenfeld8}, and refinements of these methods have been announced for $9$ and $10$ runners \cite{trakulthongchai910,rosenfeld9}. We remark that if the conjecture is to be true, the lower bound is sharp since there are known cases of equality. These cases are often referred as \emph{tight speed sets}, and a trivial construction is setting $v_i = i$ for all $1 \leq i \leq n$. However, there are more erratic cases of equality, as demonstrated by the work of Goddyn and Wong \cite{Goddyn2006TIGHTIO}. While the problem was first asked in relation to the field of diophantine approximation, it is connected to many other fields, including but not limited to: distance graphs' chromatic number \cite{10.11650/twjm/1500404981}, flows in graphs and matroids \cite{BIENIA19981}, and view-obstruction problems \cite{cusick}.

\par We focus on the “near-tight’’ regime $\,[1/(n+1),\,1/n)$. 
Assuming the Lonely Runner Conjecture, any configuration with at most $n-1$ moving runners already guarantees a gap of at least $1/n$.  Hence gaps strictly below $1/n$ can occur only when \emph{all $n$ moving runners are present} (equivalently, in the original formulation with $n+1$ runners on the circle, all $n+1$ velocities are pairwise distinct).
 Here, Kravitz questions whether the spectrum of achievable loneliness values within this interval is dense, to which he conjectures that, contrary to expectation, it is discrete \cite{kravitz}. He then proposes the following conjecture to further sharpen and formalize this observation.
\begin{conj}[Loneliness Spectrum Conjecture \cite{kravitz}]
\label{con1}
    For any positive integers $v_1,\ldots,v_n$, we have either
    $$\exists s \in \mathbb{N},  \ML(v_1,\ldots,v_n) = \frac{s}{ns + 1} \text{ or } \ML(v_1,\ldots,v_n) \geq \frac{1}{n}.$$
\end{conj}
Note that this conjecture is strictly stronger than the original Lonely Runner Conjecture. Kravitz proved the conjecture for $n = 3$. However, our paper reveals that for $n$ with greater values, there exist counterexamples like $\ML(8,3,11,19) = 7 / 30$ for $n = 4$, $\ML(5,6,11,17,23,28) = 8 / 51$ for $n = 6$, and $\ML(1,3,4,5,7,13,18) = 3 / 23$ for $n = 7$. Nonetheless, our experimental data indicate that the intuition of a discrete spectrum remains true. Thus, we propose a new modified version of the Loneliness Spectrum Conjecture.
\begin{conj}[Amended Loneliness Spectrum Conjecture]
\label{con2}
    For any positive integers $v_1,\ldots,v_n$, we have either
    $$\exists s,k \in \mathbb{N},k \leq n,  \ML(v_1,\ldots,v_n) = \frac{s}{ns + k} \text{ or } \ML(v_1,\ldots,v_n) \geq \frac{1}{n}.$$
\end{conj}
\par We refer to Conjecture \ref{con1} as the \emph{Loneliness Spectrum Conjecture} and to Conjecture \ref{con2} as the \emph{Amended Loneliness Spectrum Conjecture}.  Throughout, we may (and do) restrict attention to times of the form 
$$t=\tfrac{m}{v_i+v_j}$$
with $m\in\mathbb{Z}$ and $1\le i<j\le n$. This reduction is standard (see, e.g., the discussion around \cite[Prop.~2.1]{kravitz} and Lemma \ref{lem1}).

\begin{itemize}
  \item In particular, arguments of Kravitz show that at such times one obtains the lower bound
  \[
     \ML(v_1,\ldots,v_n)\ \ge\ \frac{\lfloor (v_i+v_j)/n\rfloor}{\,v_i+v_j\,}.
  \]
  Writing \(v_i+v_j = ns+k\) with integers \(s\ge 0\) and \(0\le k<n\) yields the value $s/(sn+k)$, which in turn recovers the baseline $\ML \geq 1/n$ and suggests a discrete family of denominators of the form \(sn+k\).
  \item In fact, we proved in Section \ref{sec4} that for the case of $n=4$, if there exists a pair of speeds with a greatest common factor of at least $4$, the loneliness value $\ML$ is at least $1/n$. We further show that if the greatest common divisor is $3$, then the loneliness value $\ML$ is also at least $1/n$ except when the speeds abide by the form $\ML(1,2,3,12k) = 3k/(12k+1)$. In both of these cases, the loneliness value either lies within the region above $1/n$ or forms a discrete spectrum as predicted by the Loneliness Spectrum. Note that we avoid calling this region ``non-discrete'': while $1/n$ is known to be an accumulation point of gaps, it is not known that every value above $1/n$ is an accumulation point.
 It remains to show that the conjecture continues to hold true when the greatest common divisor is at most $2$ or when speeds are pairwise co-prime.
  \item We also ran computations for \(4\le n\le 6\) (see Section \ref{sec5} for ranges and details). 
  In all tested instances, the amended form holds; moreover, the data indicate an even sharper restriction on the possible \(k\)-values, which we discuss in Section \ref{sec5}.
\end{itemize}

We now present a roadmap of the main contributions of this paper.

\section{Main Results}
In Section \ref{sec2}, we develop several new tools. We extend Kravitz's pre-jump, which converts the continuous problem into a modular framework, to prove a new lower bound for many modular setups (Lemma \ref{lem2}) that guarantees a loneliness value of $1/4$. We also build upon existing arguments for cases when one runner is much faster than the rest (Lemma \ref{lem3}) to prove a new lower bound when two runners of comparable speeds dominate (Lemma \ref{lem4}). Finally, we establish the following general lower bound when all but one speed share a common factor.

\begin{theorem}
    For $n \geq 4$, let $v_1, v_2, \ldots, v_n$ be positive integers with $\gcd(v_1,\ldots, v_n) = 1$ and $\gcd(v_1,\ldots,v_{n-1}) = g \geq 2$ such that the Lonely Runner Conjecture holds for $v_1, \ldots, v_{n-1}$. Then
    $$
    \ML(v_1,\ldots,v_n) \geq \frac{1}{n}.
    $$
\end{theorem}

In Section \ref{sec3}, we obtain our first major result: a disproof of the Loneliness Spectrum Conjecture \ref{con2}. We construct an infinite family of counterexamples, showing that for $n=4$ the parameter $k$ can equal $2$ rather than $1$ as predicted.

\begin{theorem}[See also Theorem \ref{the0}]
    For $n = 4$ and integer $s \geq 0$, we have
    $$
    \ML(8,4s + 3,4s + 11,4s + 19) = \frac{2s + 7}{8s + 30} = \frac{2s + 7}{4(2s + 7) + 2}.
    $$
\end{theorem}

In Section \ref{sec4}, we focus on the Amended Loneliness Spectrum Conjecture \ref{con2}, allowing $k$ to range between $1$ and $n$. We then prove the conjecture in a broad class of cases for $n=4$, when two runners share a sufficiently large common factor. More precisely:

\begin{theorem}[See also Theorem \ref{the3}, \ref{the4} and Proposition \ref{prop1}]
    Let $v_1, v_2, v_3, v_4$ be positive integers with $\gcd(v_1,v_2,v_3,v_4) = 1$. If there exists a pair of speeds that share a common factor $g > 3$, then
    $$
    \ML(v_1,v_2,v_3,v_4) \geq \frac{1}{4}.
    $$
    Likewise, if there exists a pair of speeds that share a common factor $g = 3$ and $(v_1,v_2,v_3,v_4) \not= (1,2,3,12k)$, then
    $$
    \ML(v_1,v_2,v_3,v_4) \geq \frac{1}{4}.
    $$
    Finally, if $(v_1, v_2, v_3, v_4) = (1,2,3,12k)$,
    $$
    \ML(v_1,v_2,v_3,v_4) = \frac{3k}{4(3k) + 1}.
    $$
\end{theorem}

The first two cases fall under the region above $1/n$, while the final case provides an infinite family within the discrete spectrum $[1/(n+1),1/n)$, consistent with the amended conjecture.

Finally, in Section \ref{sec5}, we present computer-assisted experiments for $n \in \{4,5,6,7,8\}$ (Table \ref{table1}) and discuss the incomplete appearances of $k$. These motivate a stronger conjecture \ref{con4}, which restricts the parameter $k$ to at most $n/2$.

\section{Preliminaries}
\label{sec2}
\par We now present the preliminaries and assumptions that all subsequent discussions rely on. These are mostly standard reductions and folklore observations; we include them for completeness and to make the paper self-contained.

\begin{lemma}[Bohman, Holzman, and Kleitman \cite{bohman}]
\label{lem5}
    Let $0 < \delta < \tfrac{1}{2}$ such that $\ML(u_1,\ldots,u_{n - 1}) > \delta$ for all choices of positive real speeds $u_1,\ldots,u_{n - 1}$. Then for positive real speeds $v_1,\ldots,v_n$, if there exists a pair of speeds with irrational ratio, i.e.\ $\tfrac{v_i}{v_j} \notin \mathbb{Q}$, we have $\ML(v_1,\ldots,v_n) > \delta$.
\end{lemma}

\par This reduction from real to rational speeds is well known (see also Henze--Malikiosis \cite{henze2017} for a related proof without appealing to the inductive LRC assumption). In particular, since the Lonely Runner Conjecture guarantees a loneliness value of at least $1/n$ for $n-1$ runners, we immediately deduce the same lower bound whenever some ratio $v_i/v_j$ is irrational.

\begin{fact}[Scaling invariance]
    For any nonzero $l\in\mathbb{R}$ and any set of integer speeds $v_1,\ldots,v_n$, we have
    \[
       \ML(v_1,\ldots,v_n) = \ML(l v_1,\ldots,l v_n).
    \]
\end{fact}

\par This invariance under dilations is a simple folklore observation. As a consequence, we may always normalize so that $\gcd(v_1,\ldots,v_n)=1$.

\begin{fact}[Reduction to integers]
    Assuming the Lonely Runner Conjecture holds for any $n-1$ positive real speeds and the Amended Spectrum Conjecture holds for $n$ integer speeds, then it also holds for $n$ real speeds.
\end{fact}

\begin{cor}
    If the Amended Spectrum Conjecture holds for all $n$-tuples of integers with $\gcd(v_1,\ldots,v_n)=1$, then it holds for all real speeds.
\end{cor}

\begin{proof}
    If some ratio $v_i/v_j$ is irrational, Lemma~\ref{lem5} implies $\ML(v_1,\ldots,v_n)\ge 1/n$. Otherwise, all ratios are rational, and clearing denominators by multiplying all speeds by their least common multiple reduces them to integer speeds. By assumption, the spectrum conjecture holds in that case.
\end{proof}

\par Finally, we recall a useful folklore lemma:

\begin{lemma}[Folklore; see Czerwiński--Grytczuk \cite{czerwinski2008}, Kravitz \cite{kravitz}]
\label{lem1}
    Let $v_1,\ldots,v_n$ be positive integers $(n \geq 2)$ with $\gcd(v_1,\ldots,v_n) = 1$. Then every local maximum of
    \[
       f(t) = \min_{1 \leq i \leq n} \|t v_i\|
    \]
    occurs at a time of the form
    \[
       t_0 = \frac{m}{v_i + v_j}
    \]
    for some $1 \leq i < j \leq n$ and $m \in \mathbb{Z}$.
\end{lemma}

\par Intuitively, a local maximum arises when two runners are positioned symmetrically about the origin; otherwise, perturbing the time increases the minimum distance. This lemma reduces the relevant times to a finite number of candidates.

\begin{define}[Modular notation]
We introduce our notation for describing runner positions. We identify the track with the unit circle $\mathbb{T} = \mathbb{R}/\mathbb{Z}$ and place the origin at $0$. For a given gap parameter $0 \leq \delta \leq 1/2$, we call a runner \emph{valid} at time $t$ if its position $tv_i \bmod 1$ lies at least $\delta$ away from the origin. 
Equivalently, the \textbf{valid region} is the closed interval 
$$ [\,\delta,\, 1-\delta\,] \subset \mathbb{T}, $$
while the complementary open arc 
$$ (\,1-\delta,\, \delta\,) $$
is the \textbf{invalid region}, which we occasionally also denote as $ (\,-\delta,\, \delta\,) $, containing all positions within distance $\delta$ of the origin.
\end{define}

\subsection{Very fast runners}
\par We now briefly shift our attention to a special class of setups that has been actively studied in the literature, namely when the set of speeds satisfies certain lacunarity properties. For example, Dubickas \cite{dubickas2011} verified the Lonely Runner Conjecture for $n \geq 16342$ under the assumption that consecutive speeds grow sufficiently quickly, $v_{j+1}/v_j \geq 1 + 33 \log n / n$. More generally, Peres and Schlag \cite{peres2010} analyzed lacunary sequences in the context of Diophantine approximation and chromatic numbers of distance graphs, showing how exponential growth conditions enforce positive separation from the integers. We focus here on two related conditions, where one or two of the runners are significantly faster than the rest.

\begin{lemma}[One Very Fast Runner \cite{kravitz}]
\label{lem3}
    Let $v_1 < v_2 < \cdots < v_{n - 1}$ be positive integers $(n \geq 2)$ with $\ML(v_1,\ldots,v_{n - 1}) = L$, and fix some $0 < \epsilon < L$. Then we have that
    $\ML(v_1,\ldots,v_n) \geq L - \epsilon$
    whenever
    $v_n \geq \left(\frac{L - \epsilon}{\epsilon}\right)v_{n - 1}$.
\end{lemma}

\begin{proof}
Choose $t_0$ such that $\|t_0v_i\|\ge L$ for all $1\le i\le n-1$.
We will construct $t$ with $|t-t_0|\le \frac{L-\epsilon}{v_n}$ such that
$\|tv_i\|\ge L-\epsilon$ for all $1\le i\le n$.

If $\|t_0v_n\|\ge L-\epsilon$, we may take $t=t_0$ and are done.
Otherwise, $t_0v_n\bmod 1$ lies in the invalid arc
\[
(-L+\epsilon,\;L-\epsilon)\subset\mathbb T,
\]
which has length $2(L-\epsilon)$.
Since runner $n$ moves at speed $v_n$, there exists $t_1$ with
\[
|t_1-t_0|\le \frac{L-\epsilon}{v_n}
\quad\text{and}\quad
\|t_1v_n\|=L-\epsilon.
\]

We now verify that the remaining runners remain valid at time $t_1$.
For $1\le i\le n-1$, the displacement of runner $i$ between times
$t_0$ and $t_1$ is at most $|t_1-t_0|v_i$, hence
\[
\|t_1v_i\|
\ge \|t_0v_i\|-|t_1-t_0|v_i
\ge L-\frac{L-\epsilon}{v_n}v_i.
\]
Since $v_i\le v_{n-1}$ for all such $i$, we obtain
\[
\|t_1v_i\|\ge L-\frac{L-\epsilon}{v_n}v_{n-1}.
\]

By the assumption $v_n\ge \frac{L-\epsilon}{\epsilon}v_{n-1}$,
the right-hand side is at least $L-\epsilon$.
Therefore,
\[
\min_{1\le i\le n}\|t_1v_i\|\ge L-\epsilon,
\]
which implies $\ML(v_1,\ldots,v_n)\ge L-\epsilon$.
This concludes the proof.
\end{proof}

\par This lemma is useful in dealing with cases when $v_{n-1}$ is bounded by a constant. We can manually check a finite number of instances and invoke this lemma when the $n$-th speed is sufficiently large.

\begin{remark}
\label{rem1}
    However, for there to be a finite number of instances to check when $n - 1$ of the speeds are fixed, we need $\epsilon$ in Lemma \ref{lem3} to be non-trivially greater than $0$. If the $n - 1$ speeds form a tight speed set such that $\ML(v_1,\ldots,v_{n - 1}) = 1 / n$, then the Lemma cannot ensure a loneliness value of at least $1 / n$ since $\epsilon = 0$. On the other hand, if we assume the Amended Loneliness Spectrum Conjecture \ref{con2} holds true for $n - 1$ runners and $v_1,\ldots,v_{n - 1}$ do not form a tight speed set, then $\ML(v_1,\ldots,v_{n - 1})$ must be non-trivially greater than $1 / n$ due to the discrete nature of the $[1 / n,1 / (n - 1))$ region.
\end{remark}

In the same vein, we introduce a method for dealing with $2$ very fast runners.

\begin{lemma}[Two Very Fast Runners]
\label{lem4}
    Let $v_1 < v_2 < \cdots < v_n$ be positive integers $(n \geq 4)$ with $\ML(v_1,\ldots,v_{n - 2}) \geq L$. Then $\ML(v_1,\ldots,v_n) \geq \frac{1}{n}$ if
    $$L \geq \frac{1}{n}\left(1 + \frac{3v_{n - 2}}{v_{n - 1}}\right).$$
\end{lemma}


\begin{proof}
Let $\delta=\frac{1}{n}$. Choose $t_0$ such that $\|t_0v_i\|\ge L$ for all $1\le i\le n-2$.
We will find $t$ with $|t-t_0|\le \frac{3}{nv_{n-1}}$ such that $\|tv_{n-1}\|\ge \delta$ and $\|tv_n\|\ge \delta$.

First, since the invalid arc for runner $n-1$ has length $2\delta$, by moving time forward or backward we can reach the boundary of the valid region in time at most $\delta/v_{n-1}$.
Thus there exists $t_1$ with $|t_1-t_0|\le \frac{\delta}{v_{n-1}}$ and $\|t_1v_{n-1}\|=\delta$.
Fix the direction of time so that for $t\ge t_1$ we have $t v_{n-1}\bmod 1\in[\delta,1-\delta]$ for as long as possible; along this direction runner $n-1$ remains valid for at least $(1-2\delta)/v_{n-1}$ time.

If $\|t_1v_n\|\ge \delta$, set $t=t_1$ and we are done.
Otherwise, $t_1v_n\bmod 1$ lies in the invalid arc $(1-\delta,\delta)$, which has length $2\delta$.
As $t$ increases, $t v_n\bmod 1$ moves continuously at speed $v_n$, so within time at most $2\delta/v_n$ it hits the boundary of the invalid arc and hence becomes valid.
Since $v_n>v_{n-1}$, we have $2\delta/v_n\le 2\delta/v_{n-1}$, so runner $n-1$ stays valid during this adjustment.
Therefore there exists $t_2\ge t_1$ with
\[
|t_2-t_1|\le \frac{2\delta}{v_n}\le \frac{2\delta}{v_{n-1}}
\quad\text{and}\quad
\|t_2v_{n-1}\|\ge\delta,\ \|t_2v_n\|\ge\delta.
\]
Combining,
\[
|t_2-t_0|\le |t_2-t_1|+|t_1-t_0|\le \frac{3\delta}{v_{n-1}}=\frac{3}{nv_{n-1}}.
\]

Finally, for $1\le i\le n-2$ we have
\[
\|t_2v_i\|\ge \|t_0v_i\|-|t_2-t_0|\,v_i
\ge L-\frac{3}{nv_{n-1}}\,v_{n-2}.
\]
Hence
\[
\ML(v_1,\ldots,v_n)\ \ge\ \min\!\Big(\delta,\ L-\frac{3v_{n-2}}{nv_{n-1}}\Big).
\]
If $L\ge \delta\big(1+\frac{3v_{n-2}}{v_{n-1}}\big)$, the right-hand side is at least $\delta=1/n$, proving the lemma.
\end{proof}

\par This lemma proves useful when the two fastest runners share similar speeds such that we cannot apply Lemma \ref{lem3}.

\par For the sake of completeness, we also recall Kravitz's result for $n = 3$, which establishes the original Loneliness Spectrum Conjecture and therefore also implies the Amended Loneliness Spectrum Conjecture for $n = 3$.

\begin{theorem}[Kravitz \cite{kravitz}]
\label{the5}
    For any positive integers $v_1,v_2,v_3$, we have either
    $$\ML(v_1,v_2,v_3) = \frac{s}{3s + 1} \text{ for some } s \in \mathbb{N} \text{ or } \ML(v_1,v_2,v_3) \geq \frac{1}{3}.$$
\end{theorem}

Finally, before we dive into the main results of this paper, we introduce one more technique called \emph{pre-jump}, first employed by Bienia, Goddyn, Gvozdjak, and Tarsi \cite{BIENIA19981}.

\begin{lemma}[Pre-jump \cite{BIENIA19981}]
    Given $2$ positive integer speeds $v_1$ and $v_2$. If they share a common factor of $g$, then for any real number time $t$ and integer choices of $h$, we have
    $$\min(\norm{tv_1},\norm{tv_2}) = \min\left(\Norm{\left(t + \frac{h}{g}\right)v_1},\Norm{\left(t + \frac{h}{g}\right)v_2}\right).$$
\end{lemma}

\begin{proof}
    We first note the fact that $\norm{x + k} = \norm{x}$ for any integer $k$ and $\norm{(t + h / g)v_1} = \norm{tv_1 + (h / g)v_1}$. From the assumption, we know that $v_1$ is divisible by $g$, hence the second part is an integer, attaining the equality $\norm{(t + h / g)v_1} = \norm{tv_1}$. The same argument can be applied for $v_2$, resulting in the min function taking in the same arguments on both sides of the equality.
\end{proof}

This pre-jump is especially effective when combined with the Pigeonhole Principle, as demonstrated in the following corollary:

\begin{cor}
Let $v_1,v_2,v_3$ be three positive integers with $\gcd(v_1,v_2,v_3)=1$. Assume the Lonely Runner Conjecture holds for $n=2$. If any two of the speeds share a common factor at least $2$, then the Lonely Runner Conjecture holds for these three integers.
\end{cor}

This result is classical and has several simple proofs; see, e.g., Proposition 4.6 in \cite{jana}, or by Kravitz in \cite{kravitz}.  We reproduce one here to illustrate the pre-jump, which we will use extensively in later sections.

\begin{proof}
    Without the loss of generality, let $\gcd(v_1,v_2) = g \geq 2$. Since by assumption $\gcd(v_1,v_2,v_3) = 1$, we know that $g$ is co-prime to $v_3$. Applying the Lonely Runner Conjecture for $n = 2$, there exists a time $t$ such that
    $$\min(\norm{tv_1},\norm{tv_2}) \geq \frac{1}{3}.$$
    We now apply the pre-jump and define a new time
    $t' = t + \frac{h}{g}$
    for some integer $0 \leq h < g$. Since $v_3$ is co-prime to $g$, we know that
    $$\{hv_3 \bmod g : 0 \leq h < g\}$$ must form the 
    complete residue set. Thus the set of possible positions of $(t + h / g)v_3$ on the unit circle is $g$ equidistantly spaced points shifted by an arbitrary amount. By the Pigeonhole principle, at least one point is at most $1 / (2g)$ away from the $1 / 2$ midpoint. Hence, there must exist an integer $h$ such that
    $$\Norm{\left(t + \frac{h}{g}\right)v_3} \geq \frac{1}{2} - \frac{1}{2g} \geq \frac{1}{4}.$$
    We know that such a time $t' = t + h / g$ already guarantees a loneliness value of at least $1 / 3$ for the first two runners. Therefore, the overall loneliness value is at least $1 / 4$, as desired.
\end{proof}

\par We remark that a similar argument can even be applied for more runners, although the usage of the Pigeonhole Principle is much more intricate when there is more than one runner. We discuss this further in Section \ref{sec4}, and specifically Lemma \ref{lem2}.

\section{Achieving the exceptions}
\label{sec3}
\par One of our key results is disproving the original Loneliness Spectrum Conjecture through counterexamples. Since Kravitz proved the Loneliness Spectrum Conjecture for $n = 3$ \cite{kravitz}, we look at the next unproven case of $n = 4$. A bit of experimentation suggests an explicit construction of counterexamples for four moving runners that produces loneliness values of the form
$$\frac{s}{4s + 2}$$
for odd integer $s \geq 5$. Note that if $s$ is even, the fraction reduces to the non-exceptional form, so we only focus on the odd case.

\begin{lemma}
\label{lem6}
    For $n = 4$ and natural numbers $a$ and $b$ such that $a \neq b$ and $\gcd(a,b) = 1$, we define $v_1 = a$ and $v_i = (i - 1)a + b$ for $2 \leq i \leq 4$. Then,
    $$\ML(v_1,\cdots,v_4) \leq \frac{1}{4}.$$
\end{lemma}

\begin{proof}
    For the sake of contradiction, assume that there exists $t \in \mathbb{R}$ such that $\Norm{tv_i} > \frac{1}{4}$ for all $1 \leq i \leq 4$. In other words, runners $2$, $3$, and $4$ are all between the $(\frac{1}{4},\frac{3}{4})$ invalid region, which has a length of less than $\frac{1}{2}$. At the same time, the positions of these runners form an arithmetic progression with common difference $ta = tv_1$. This implies that,
    $$\norm{tv_1} = \norm{ta} \leq \frac{1 - \frac{2}{4}}{4 - 2} = \frac{1}{4},$$
    which contradicts our assumption that $\norm{ta} > \frac{1}{n} = \frac{1}{4}$.
\end{proof}

The above lemma motivates us to look at speeds that form an arithmetic progression sequence. We analyze one such family of speeds.

\begin{theorem}
    \label{the0}
    For $n = 4$ and integer $s \geq 0$, we have
    $$\ML(8,4s + 3,4s + 11,4s + 19) = \frac{2s + 7}{8s + 30} = \frac{2s + 7}{4(2s + 7) + 2}.$$
\end{theorem}

\begin{proof}
    \par By Lemma \ref{lem1}, any local maximum of $f(t)$ occurs at a time  $t = m/(v_i+v_j)$. Thus, the denominator of the maximizing time must be one of the pairwise speed sums:
    $$ v_1+v_2 = 4s+11,\;\; v_1+v_3 = 4s+19,\;\; v_1+v_4 = 4s+27,\;\; $$
    $$ v_2+v_3 = 8s+14,\;\; v_2+v_4 = 8s+22,\;\; v_3+v_4 = 8s+30.$$
    We further note that the denominators $4s + 11$ and $4s + 19$ are impossible because $v_3 = 4s + 11$ and $v_4 = 4s + 19$. 
 Invoking Lemma \ref{lem6} with $a = 8$ and $b = 4s - 5$ tells us that the loneliness value must be less than $1/4$. Hence, the remaining denominators, the largest achievable value $< 1/4$ is $\max\{(s + 6) / (4s + 27), (2s + 3) / (8s + 14), (2s + 5) / (8s + 22), (2s + 7) / (8s + 30)\} = (2s + 7) / (8s + 30)$. It remains to prove that this is always attainable. Using the Euclidean algorithm, we know that $\gcd(4s + 19,8s + 30) = 1$, i.e. $v_4 = 4s + 19$ and $8s + 30$ are co-prime. Therefore, there must exist an integer $u$ such that
    $$uv_4 \equiv (4s + 19)u \equiv 2s + 7 \pmod {8s + 30}$$
    Let $t = u / (8s + 30)$, then $\norm{tv_4} = (2s + 7) / (8s + 30)$. We now show that $\norm{tv_1}$, $\norm{tv_2}$, and $\norm{tv_3}$ are all at least $\norm{tv_4}$. We first focus on $tv_1 = (8u) / (8s + 30)$, deriving the following equation
    $$uv_1 \equiv 8u \equiv (8 + (8s + 30))u \equiv (8s + 38)u \equiv 2(4s + 19)u \equiv 2(2s + 7)\pmod {8s + 30}.$$
    Similarly for $uv_2$, we get
    $$uv_2 \equiv (4s + 3)u \equiv (4s + 19)u - 16u \equiv (2s + 7) - 4(2s + 7) \equiv -3(2s + 7) \pmod {8s + 30}.$$
    Finally, for $uv_3$, we get
    $$uv_3 \equiv (4s + 11)u \equiv (4s + 19)u - 8u \equiv (2s + 7) - 2(2s + 7) \equiv -(2s + 7) \pmod {8s + 30}.$$
    We recall that $\norm{x} = \norm{-x}$, so we can take the absolute value of these values. We are left with
    $$\min\left(\norm{tv_1},\norm{tv_2},\norm{tv_3},\norm{tv_4}\right) = \min\left(\Norm{\frac{2(2s + 7)}{8s + 30}},\Norm{\frac{3(2s + 7)}{8s + 30}},\Norm{\frac{(2s + 7)}{8s + 30}},\Norm{\frac{(2s + 7)}{8s + 30}}\right).$$
    Straightforward comparisons show that all the arguments are at least $(2s + 7) / (8s + 30)$, as desired
    
\end{proof}

 \par We further note that we never assumed $s \geq 0$. By setting $s = -1$. we get the construction for $\ML(8, -1, 7, 15) = \ML(8, 1, 7, 15) = \frac{5}{22}$. Finally, we have not found any instances of $\ML = \frac{3}{14}$. $\ML = \frac{1}{6}$ does not exist as it contradicts the Lonely Runner Conjecture's $\frac{1}{5}$ lower-bound.
\par Another peculiar result is that we have only found instances of loneliness value $s / (4s + k)$ for $k \in \{1,2\}$ in our experiments. Under the Conjecture, there could potentially be instances of $k \in \{3,4\}$, but none are present in our experiments where all speeds are at most $400$. Thus, we propose the following Conjecture.
\begin{conj}
\label{con3}
    Given $4$ positive integer speeds $v_1,v_2,v_3,v_4$, we have either
    $$\exists s,k \in \mathbb{N},k \in \{1,2\}, \ML(v_1,\ldots,v_4) = \frac{s}{4s + k} \text{ or } \ML(v_1,\ldots,v_4) \geq \frac{1}{4}.$$
\end{conj}
We remark that this incomplete appearance of the values of $k$ is not unique to $n = 4$, but some variant of it seems to also be true for at least $n = 5$ and $n = 6$. We explore this more in Section \ref{sec5}

\section{Large common factors}
\label{sec4}
In this section, we study the loneliness value when the speeds share a sufficiently large common factor. Specifically, we illustrate that in the case of four lonely runners satisfying certain conditions on common factors, the setup inherently produces a loneliness value of at least $\frac{1}{4}$. 

\subsection{More techniques}
\par We first introduce a few more tools. All the results within this section are concerned with the region above $1/n$ in the Spectrum Conjecture. Therefore, we aim for inequalities of the form $\ML(v_1, \ldots, v_n) \geq 1/n$ instead of the usual $1/(n+1)$ as in the original Conjecture.
\begin{theorem}
\label{the1}
    For $n \geq 4$, let $v_1, v_2, \ldots, v_n$ be positive integers with $\gcd(v_1,\ldots, v_n) = 1$ and $\gcd(v_1,\ldots,v_{n-1}) = g \geq 2$ such that the Lonely Runner Conjecture holds for $v_1, \ldots, v_{n-1}$. Then
    $$\ML(v_1,\ldots,v_n) \geq \frac{1}{n}.$$
\end{theorem}

\begin{proof}
    Since we assume that the $n$ speeds do not share a common factor, $v_{n}$ and $g$ are co-prime. The Lonely Runner Conjecture for $n-1$ runners dictates that there exists a time $t$ such that
    $$\min(\|tv_1\|,\ldots,\|tv_{n-1}\|) \geq \frac{1}{n}.$$
    Subsequently, we apply the pre-jump, defining a new time
    $$t' = t + \frac{h}{g}$$
    Which fixes the positions of the first $n-1$ runners while moving the $n$-th runner in $\frac{1}{g}$ intervals. By the Pigeonhole Principle, there exists an integer $h$ such that
    $$\left\lVert\left(t + \frac{h}{g}\right)v_n\right\rVert \geq \frac{1}{2} - \frac{1}{2g} \geq \frac{1}{n}$$
    Thus we conclude that $\ML(v_1,\ldots,v_n) \geq \frac{1}{n}$.
\end{proof}

For the rest of the paper, we refer to Theorem \ref{the1} only for $n = 4$. The theorem allows us to restrict our focus speeds with at most two runners sharing a common factor. Another useful lemma that we will repeatedly utilize throughout the paper applies when two speeds share a common factor $g \geq 6$ and satisfy a couple more conditions. This lemma is very powerful when combined with pre-jumps.
\begin{lemma}
    \label{lem2}
    Let $0\le a,b\le 1$ be real numbers and let $v_1,v_2,g$ be positive integers. If $g \geq 6$, $\gcd(v_1,g) = 1$, $v_2 \not\equiv 0 \bmod g$, and $v_1 \neq \pm v_2 \bmod g$, then there exists integer $0 \leq h < g$ such that
    $$\min\left(\left\lVert\frac{h}{g}v_1 + a\right\rVert,\left\lVert\frac{h}{g}v_2 + b\right\rVert\right) \geq \frac{1}{4}.$$
\end{lemma}

\begin{proof}
    The pre-jump effectively transforms the problem into modular arithmetic for modulo $g$. Let $v_1 \equiv x \bmod g$ and $v_2 \equiv y \bmod g$. Since $v_1$ is co-prime to g, there is a modular inverse $u$ such that $ux \equiv 1 \bmod g$. We define $uy \equiv uv_2 \equiv z \bmod g$, then by assumption, $z \neq \pm 1$, and $z \neq 0$ since $v_2$ is not divisible by $g$.

    We increment $h$ by $u$ at a time. Let $h_1$ be some time such that $\left((h_1v_1)/g + a \bmod 1 \right) \geq \frac{1}{4}$ but $h_0 = h_1 - u$ would push runner one into the invalid region, i.e. $< 1/4$. Note that we intentionally left out the norm sign as we want the signed value in $\mathbb{R}/\mathbb{Z}$. We define $h_i = h_1 + u k$ for integers $i > 1$. Then, since $u v_1 \equiv 1 \bmod g$, as we increase $i$ by one, $\left((h_2v_1)/g + a \bmod 1\right)$ increases by $\frac{1}{g}$. We eventually reach some $k + 1$ such that $\left((h_{k+1}v_1)/g + a \bmod 1 \right) > \frac{3}{4}$ and runner one lands in the invalid region. In the worst case, there are at least $\Floor{\frac{1}{2} / \frac{1}{g}} = \Floor{\frac{g}{2}}$ of these $h_i$ values satisfy the conditions, meaning that $k \geq \Floor{\frac{g}{2}}$. 
    
    We now show that at least one of these $k$ values of $h_i$ also satisfy the inequality for runner two. We first assume for convenience that $uy \equiv z \bmod g$ is at most $g / 2$. Otherwise, we can negate both $v_2$ and $b$ while preserving all of our conditions. When $h = h_1$, runner two is either already within the valid region, in which case we're done, or somewhere within the invalid region, which has a length strictly less than $1/2$. Since $z \leq g / 2$, as we increment $i$, runner two can never skip over the entirety of the valid region. Furthermore, considering runner two's $g$ possible equidistant positions on the unit circle. Since $z \not\equiv \pm 1$, runner two must jump by at least two positions at a time. The invalid region contains at most $g - \floor{g/2}$ points. Hence, runner two should pass by at least $2k - 1$ valid points as we iterate from $h_1$ to $h_k$. Substituting $k \geq \Floor{g / 2}$, we derive the inequality
    $$2k - 1 > g - \floor{g/2} \implies 3\Floor{g / 2} - 1 > g$$
    which is true for all $g \geq 6$.
\end{proof}

\subsection{Proving three shifted runners}
A nice bonus that we get from applying the aforementioned methods is proving the Shifted Lonely Runners Conjecture for $n=3$. The existing proof by Cslovjecsek involves a setup of geometric lattices and the covering radius of a polytope \cite{jana}. We instead invoke only elementary arithmetics. This theorem actually comes in handy later in our other proofs as well.





\begin{theorem}
\label{the2}
    Given distinct positive integers $v_1, v_2, v_3$ with $\gcd(v_1,v_2,v_3) = 1$ and arbitrary starting points $s_1, s_2, s_3 \in \mathbb{R}$, there is a real number $t$ such that for all $1 \leq j \leq 3$, the distance of $s_j + tv_j$ to the nearest integer is at least $\frac{1}{4}$.
\end{theorem}

\begin{proof}
    Suppose that two of the speeds, say $v_i$ and $v_j$, share a common factor $g>1$, and let $v_k$ denote the remaining third speed. It is known that for $d$ moving runners with arbitrary starting positions, the loneliness value is at least $1/(2d)$ \cite{SCHOENBERG1976263} (see also Theorem~4 in \cite{BeckHostenSchymura2019}). In particular, for $d=2$ there exists a time $t$ such that
    \[
    \min\bigl(\|tv_i+s_i\|,\ \|tv_j+s_j\|\bigr)\ge \tfrac14.
    \]
    We now apply the pre-jump and define a new time $t' = t + \frac{h}{g}$ for some integer $h$.

    Since $g$ and $v_k$ are co-prime by assumption, the new time fixes runner $i$ and $j$ while moving runner $k$ in $\frac{1}{g}$ intervals. Finally, by the Pigeonhole Principle, there exists an integer $h$ such that
    $$\Norm{\left(t + \frac{h}{g}\right)v_k + s_k} \geq \frac{1}{2} - \frac{1}{2g} \geq \frac{1}{4}$$
    We conclude that the Shifted Lonely Runner Conjecture holds if there exist two speeds that share a common factor.
    
    Now we focus on the case when the speeds are pairwise co-prime. We restrict the positions of runners $2$ and $3$ by only focusing on time values $t$ such that
    $$(tv_2 + s_2) + (tv_3 + s_3) = n$$
    for some integer $n$. Intuitively, this restricts runners $2$ and $3$ to always be symmetric in terms of their positions, so we can effectively ignore one of the runners (in this case runner $3$). More formally
    $$(tv_2 + s_2) \equiv -(tv_3 + s_3) \bmod 1$$
    Simplifying the original condition gives
    $$t(v_2 + v_3) = (n - s_2 - s_3) \implies t = \frac{(n' - s_2 - s_3)}{v_2 + v_3} + \frac{h}{v_2 + v_3}$$
    for some integer $h$ and $n = n' + h$. Without the loss of generality, let $v_1 < v_2 < v_3$. Consequently, $v_1 \neq \pm v_2 \bmod {v_2 + v_3}$. So if $v_2 + v_3 \geq 6$, the conditions of Lemma \ref{lem2} are satisfied with $g = v_2 + v_3$, yielding a loneliness value of at least $\frac{1}{4}$. It remains to deal with the case when $v_2 + v_3 < 6$. We note that the only triple of speeds that fits this description is $(1,2,3)$. Using the same restriction method again, but this time on runners $1$ and $3$, the formula produces time values of the form
    $$t = \frac{n - s_1 - s_3}{4} + \frac{h}{4}.$$
    We note that the $t$ values for $h = 2$ and $h = 1$ are exactly $\frac{1}{4}$ apart, meaning the positions of runner $2$ at those times are $\frac{1}{4} \cdot 2 = \frac{1}{2}$ apart, thus one of them must produce a loneliness value of at least $\frac{1}{4}$ for runner $2$ using the same Pigeonhole principle argument, which we denote as $h'$. Finally, we can tweak $h'$ by choosing whether or not to add $2$ to $h'$. This does not change runner $2$'s position because $2 \cdot \frac{2}{4} = 1$, but it does change runner $1$'s position by $1 \cdot \frac{2}{4} = \frac{1}{2}$, so one of $h'$ and $h' + 2$ must produce a loneliness value of at least $\frac{1}{4}$ for runners $1$ and $3$, as desired. We conclude that the loneliness value is at least $\frac{1}{4}$.
\end{proof}

\subsection{Large common factor between two speeds}
In the previous subsections, specifically Theorem \ref{the1}, we proved that when three of the speeds share a common factor, the loneliness values fall within the ``non-discrete'' region of at least $1/4$ predicted by the Loneliness Spectrum Conjecture. In this section, we further prove the loneliness value is at least $\frac{1}{4}$ if any two speeds share a common factor $g > 3$. 

\begin{theorem}
    \label{the3}
    Let $v_1, v_2, v_3, v_4$ be positive integers with $\gcd(v_1,v_2,v_3,v_4) = 1$. Then if any two of the speeds share a common factor $g > 3$, we have
    $$\ML(v_1,v_2,v_3,v_4) \geq \frac{1}{4}.$$
\end{theorem}

\begin{proof}
    Without the loss of generality, let runners $1$ and $2$ share a common factor $g > 3$. Furthermore, if runners $3$ and $4$ share a common factor with $g$, then Theorem \ref{the1} guarantees a loneliness value of at least $1/4$, so we assume $v_3$ and $v_4$ are both not divisible by $g$.
    
    At a high level, we roughly follow the three strategies below for dividing the casework:
    \begin{enumerate}[label=(\roman*)]
        \item For $g \geq 6$, we can first apply the pre-jump fixing runners $1$ and $2$. We then invoke Lemma \ref{lem2} for runners $3$ and $4$ under modulo $g$ to finish the proof. This comes with the caveat that $v_3 \not\equiv v_4 \pmod g$.
        \item If $v_3 \equiv v_4 \pmod g$, then $v_3$ and $v_4$ move in lock-step under modulo $g$. We reformulate the problem with three new speeds $(v_1, v_2, |v_3 - v_4|)$ starting at $(0,0,0.5)$. The theorem for the three Shifted Lonely Runners allows us to satisfy runner $1$ and $2$ while keeping runner $3$ and $4$ close to each other. Finally, we use the pre-jump to find some appropriate $h / g$ values for runners $3$ and $4$.
        \item For $g \in \{4, 5\}$, the grid is too coarse for the pre-jump, so we adopt ad-hoc strategies instead. Specifically for $g = 5$, we leverage lemmas \ref{lem3} and $\ref{lem4}$ for One and Two Very Fast Runners by bounding $v_1$ and $v_2$.
    \end{enumerate}

    More formally, we look at the following cases:
    \begin{itemize}
        \item $v_3 \neq \pm v_4 \bmod g$ and $g \geq 6$. We satisfy runners $1$ and $2$ with the $2$-runner setup. There exists some $t$ such that
        $$\min(\norm{tv_1},\norm{tv_2}) \geq \frac{1}{3} > \frac{1}{4}.$$
        We apply the pre-jump and select times of the form
        $$t' = t + \frac{h}{g}$$
        for some integer $h$. In this case, the requisite conditions are met for Lemma \ref{lem2} where $g = g$, so there exists integer $h$ such that
        $$\min\left(\Norm{\left(t + \frac{h}{g}\right)v_3},\Norm{\left(t + \frac{h}{g}\right)v_4}\right) \geq \frac{1}{4}.$$
        And due to the nature of the pre-jump, we have $\norm{(t + h / g)v_1} = \norm{tv_1}$ and $\norm{(t + h / g)v_2} = \norm{tv_2}$, both of which are at least $\frac{1}{4}$ by how we chose the value for $t$. Hence, we conclude that the loneliness values are at least $\frac{1}{4}$.
        \item $v_3 \equiv \pm v_4 \bmod g$, $|v_3 - v_4| \not\in \{v_1, v_2\}$, and $g > 3$. We focus on times of the form
        $$t' = t + \frac{h}{g}$$        
        We assume $v_3 \equiv v_4 \pmod g$ because we can otherwise flip the sign of $v_4$ to attain the desired conditions. We note that since the two runners move at the same speed modulo $g$, the distance between the two runners remains constant regardless of the value of $h / g$. For $g$ equidistant residue positions on the unit circle that are arbitrarily shifted, the position between $\frac{1}{4}$ and $\frac{3}{4}$ in the valid region with the smallest value is at most $\frac{1}{4} + \frac{1}{g}$. Otherwise, we can subtract $\frac{1}{g}$ to obtain a smaller reachable position. This implies that we can always fit a contiguous region between $\frac{1}{4}$ and $\frac{3}{4}$ where the smaller end starts at one of the $g$ positions if its length $l$ meets the following inequality
        $$\frac{1}{4} + \frac{1}{g} + l \leq \frac{3}{4}.$$
        Rearranging the terms gives us
        $$l \leq \frac{1}{2} - \frac{1}{g}.$$
        In other words, we need to find some time $t$ such that
        $$\Norm{t\left(v_3 - v_4\right)} \leq \frac{1}{2} - \frac{1}{g},$$
        which means that both runners $3$ and $4$ are $\norm{t(v_3 - v_4)}$ apart but are both still in the valid region. Note the usage of the less than or equal sign, as opposed to the usual greater than or equal sign. We can convert it to the familiar form by shifting by $0.5$
        $$\Norm{t\left(v_3 - v_4\right) + 0.5} \geq \frac{1}{2} - \left(\frac{1}{2} - \frac{1}{g}\right) = \frac{1}{g}.$$
        Since $|v_3 - v_4| \neq v_1$ or $v_2$, we can use theorem \ref{the2}, the theorem for three Shifted Lonely Runner for $3$ moving runners with speeds $(v_1,v_2,|v_3 - v_4|)$ and starting positions are $(0,0,0.5)$. This gives us a time $t$ such that runners $1$ and $2$ are at least $\frac{1}{4}$ away from the origin, and runners $3$ and $4$ are at most $\frac{1}{4}$ apart, which is less than or equal to $\frac{1}{2} - \frac{1}{g}$ as desired. 
        
        \item $v_3 \equiv \pm v_4 \pmod g$, $|v_3 - v_4| \in \{v_1, v_2\}$, and $g \geq 6$. Unfortunately, the proof for the case above breaks down when $|v_3 - v_4| = v_3$. Without the loss of generality, let $|v_3 - v_4| = v_1$. The shifted setup now has speeds $(v_1, v_2, v_1)$ with positions $(0, 0, 0.5)$. However, since runners $1$ and $3$ have the same velocities, they will always be diametrically opposed to each other. Consequently, the loneliness value is at most $1/4$ even if we ignore runner $2$ when runners $1$ and $3$ are at $1/4$ and $3/4$. Thus, we give up on $1/4$ and settle for $\frac{1}{g}$ instead. More formally, we show that there exists some real number $t$ such that
        \begin{align*}\min(\norm{tv_1},\norm{tv_2}) \geq \frac{1}{4} \quad \text{and} \quad \norm{tv_1 + 0.5} \geq \frac{1}{g}.\end{align*}
        
        \par We assume the new speeds $(v_1, v_2, |v_3 - v_4|)$ or $(v_1, v_2, v_1)$ have already been normalized, i.e., divided by their common factors, which in this case means $\gcd(v_1,v_2) = 1$. At time $t = \frac{1}{4v_1}$, we have $\min(\Norm{tv_1},\Norm{tv_1 + 0.5}) = \frac{1}{4}$. Now, if $v_1 > 1$, we can again apply the pre-jump with $t' = t + \frac{h}{v_1}$. By the Pigeonhole Principle, there exists an integer $h$ such that $\Norm{\left(t + h / v_1\right)v_2} \geq \frac{1}{4}$, as desired. The other case then is when $v_1 = 1$, a.k.a. speeds in the form of $(v_1,v_2,|v_3 - v_4|) = (1,v_2,1)$. If $v_2$ is not a multiple of $4$, $t = \frac{1}{4}$ yields the desired loneliness value. Otherwise, we assume that $v_2$ is a multiple of $4$. If we satisfy the conditions for runners $1$ and $3$ by setting $t = \frac{1}{4}$, runner $2$ would still be at the origin $0$ and can reach the $\frac{1}{4}$ mark within $\frac{1}{4} / v_2$ units of time. Setting $t = 1/4 + 1/(4v_2)$ allows runner $3$ to at least remain 
        $$\Norm{\left(\frac{1}{4} + \frac{1}{4v_2}\right) \cdot 1 + 0.5} = \frac{1}{4} - \frac{1}{4v_2} \geq \frac{1}{4} - \frac{1}{16} \geq \frac{1}{g}$$
        away from the origin. Hence, we conclude that the loneliness value is always at least $\frac{1}{4}$ if $g > 5$. Now, we handle $g = 5$ and $g = 4$ separately.
    \end{itemize}

    \begin{itemize}
        \item $g = 4$: First, the only possible pairs of remainders for $v_3$ and $v_4 \bmod g$ are $(1,1)$, $(1,3)$, $(3,1)$, or $(3,3)$ since $g = \gcd(v_1, v_2)$ can be assumed to be co-prime to $v_3$ and $v_4$ in view of Theorem \ref{the1}. All of these satisfy $v_3 \equiv \pm v_4 \bmod g$, so we will borrow much of the same proof from the previous discussion. Our second case already covers $|v_3 - v_4| \neq v_1$ or $v_2$. Thus, the only case we must take care of is when $|v_3 - v_4| = v_1$. The previous proof already gives a loneliness of at least $\frac{1}{4}$ if $v_1 = |v_3 - v_4| \geq 2g$ (when $v_1 > 1$ before normalizing the speeds in the third case), so we assume $v_1 \leq g$. Since $v_1$ is divisible by $g$, it has to be $g = 4$, so we focus only on speeds satisfying the following conditions
        \begin{align*}(v_1,v_2,|v_3 - v_4|) = (g,v_2,g) \quad \text{and} \quad v_2 \equiv 0 \bmod {g}.\end{align*}
        Transforming this setup back to the original sets of speed, we have the quadruple $(v_1,v_2,v_3,v_4) = (4,4a,b + 4,b)$ for some positive integers $a$ and $b$. If $b + 4 \geq 6$, we select times of the form
        $$t = \frac{k}{v_3}$$
        for some integer $k$. Since $v_1 = 4$, it is co-prime to $v_3 = b + 4$, thus Lemma \ref{lem2} applies (Note that because there are no shifts in this scenario, the case $v_1 \equiv \pm v_2 \bmod {v_3}$ doesn't actually matter as they would always share these same positions, so we can simply ignore one of the players), implying there exists some integer $k$ such that $\min\left(\Norm{\frac{k}{v_3}v_1},\Norm{\frac{k}{v_3}v_2}\right) \geq \frac{1}{4}$. We define this time as $t = \frac{k}{v_3}$. The intuition behind this construction is that runner $3$ stays at the origin regardless of $t$, thus when we do the pre-jump of $t' = t + \frac{h}{4}$, there are three valid $h$ values for runner $3$: $1,2,3$. Since the positions are evenly spread out, there must be at least $4 / 2 = 2$ valid $h$ values for runner $4$. Because $2 + 3 = 5 > 4$, by the Pigeonhole Principle, at least one of the $h$ values is valid for both runners $3$ and $4$, and, due to the definition of $t$ and the pre-jump, also valid for the first two runners. Hence, we conclude that for $v_3 \geq 6$, the loneliness value is at least $\frac{1}{4}$. The only case not covered is when $v_3 < 6$, which is only possible for the following speeds
        $$(v_1,v_2,v_3,v_4) = (4,4a,5,1).$$
        The loneliness value between $(1,4,5)$ is $\frac{1}{3}$, so by applying Lemma \ref{lem3} with $L = \frac{1}{3}$ and $\epsilon = \frac{1}{12}$, we have $\ML(4,4a,1,5) \geq \frac{1}{3} - \frac{1}{12} = \frac{1}{4}$, as desired for $a > 3$. We manually verify the results for $\ML(4,8,1,5)$ and $\ML(4,12,1,5)$ to cover the rest.

        \item $g = 5$: First, we focus on when $v_3 \equiv \pm v_4 \bmod g$. The second case suffices when $|v_3 - v_4| \not\in \{v_1, v_2\}$. Thus, we assume $|v_3 - v_4| = v_1 = g$ and $v_2$ is a multiple of $g$. Dividing the speeds by $g$, we get the following triple of speeds
        $$(v_1,v_2,|v_3 - v_4|) / 5 = (1,k,1)$$
        for some integer $k$. When $k$ is sufficiently large, we can use a very similar strategy as Lemma \ref{lem3} (Note that the two scenarios are not completely equivalent as two runners need to be $\frac{1}{4}$ away while the last runner needs to be $0.3$ away) to prove there exists time $t$ that guarantees $|v_3 - v_4| \leq 0.3$ if $k > 8$. We omit the exact details here as they are very redundant, but intuitively when the second runner is fast enough, we can satisfy runners $1$ and $3$ and perturb the time slightly to also satisfy runner $2$. Thus, we treat the case for when $k \leq 8$. Transforming this setup back to the original sets of speed, we have the quadruple $(v_1,v_2,v_3,v_4) = (5,5k,b + 5,b)$ for some positive integer $b$. Finally, when $b$ is sufficiently large, we can apply Lemma \ref{lem4} (Two Very Fast Runners) with $L = \ML(5,5k) \geq \frac{1}{3}$ to attain a loneliness value of at least $\frac{1}{4}$ for some sufficiently large $b$. We have to at most check for $b$ up to $1/(1/3 - 1/4)(5 \cdot 8) \cdot 3 = 1440$. This finally reduces the uncovered quadruple of speeds to a finite set, which we manually verified with a computer-assisted algorithm.
        \par It remains to prove the case when $v_3 \neq \pm v_4 \bmod g$. Without the loss of generality, we assume $v_3 > v_4$. We observe that Lemma \ref{lem2} only fails to yield the desired loneliness for $g = 5$ if only $2$ out of the $5$ values of $h$ satisfy
        $$\Norm{\frac{h}{g}v_3 + a} \geq \frac{1}{4}.$$
        If there are $3$ valid positions, then runner $3$ would have $3$ points in the valid region. Since the other runner is at least twice as fast as runner $3$, it would pass by $2 \cdot 3 - 1 = 5$ points, so it must step into the valid region at some point. Thus, we just need to prove there exists time $t$ such that $\min(\norm{tv_1},\norm{tv_2}) \geq \frac{1}{4}$ and the inequality $\Norm{\left(\frac{h}{5} + t\right)v_1} \geq \frac{1}{4}$ has at least $3$ valid $0 \leq h < 5$ solutions. We achieve this by looking for times of the form
        $$t = \frac{2k + 1}{4v_3} = \frac{k}{2v_3} + \frac{1}{4v_3}.$$
        for some integer $k$ that would satisfy runners $1$ and $2$. For $h = 0$, this fixes runner $3$'s positions at either $\frac{1}{4}$ or $\frac{3}{4}$, both yielding the desired $3$ valid solutions for $h$. We know that from Lemma \ref{lem2}, where we apply $g = 2v_3$, there exists some integer $k$ that satisfies runners $1$ and $2$ as long as $2v_3 = g \geq 6$. However, the lemma also requires that $2v_3$ is co-prime with at least one of $v_1$ or $v_2$, which is not guaranteed. We can remedy this by dividing $2v_3$ by its greatest common divisor with one of the speeds, which we now demonstrate must be less than $4$. If the greatest common factor between $v_3$ and $v_1$ or $v_2$ is $4$ or $\geq 6$, then we have already proven those cases in the previous discussion. If the greatest common factor is $5$, then $v_3$, $v_1$, and $v_2$ would share a common factor of $5$, yielding a loneliness value of at least $\frac{1}{4}$ by Theorem \ref{the1}. So the greatest common factors between $v_3$ and $v_1$ or $v_2$ is $1$, $2$, or $3$. We further note that $v_1$ and $v_2$ can at most each take one of the $2$ or $3$ factor, otherwise they would have a gcd of at least $6$, which again is already covered. Thus $$\min(\gcd(2v_3,v_1),\gcd(2v_3,v_2)) \leq \min(3,2 \cdot 2) \leq 3.$$
        We apply Lemma \ref{lem2} with the reduced modulus 
        $$g' = \frac{2v_3}{\min(\gcd(2v_3,v_1),\gcd(2v_3,v_2))} \geq \frac{2v_3}{3} \geq 6$$
        which holds if $v_3 \geq 9$. Thus, Lemma \ref{lem2} should suffice for $v_3 \geq 9$. We now do casework based on the value of $v_3$ for $1 \leq v_3 < 9$:
        \begin{itemize}
            \item[--] $v_3 = 8$: Since the factor of $3$ can't divide $8$, at most one of the speeds can share a factor of $2$, implying that $v_3$ is coprime with at least one of the speeds, i.e. 
            $$\min(\gcd(2v_3, v_1), \gcd(2v_3, v_2)) = 1,$$ satisfying the condition that $2v_3 = 16 \geq 6$.
            \item[--] $v_3 = 7$: $v_3$ must be co-prime to both speeds, satisfying the condition that $2v_3 = 14 \geq 6$.
            \item[--] $v_3 = 6$: The possible pairs of $(v_3,v_4)$ are $(6,1)$ -- falls into the proven $v_3 \equiv \pm v_4 \bmod g$ case; $(2,6)$ and $(4,6)$ -- since neither $v_1$ nor $v_2$ can share the factor of $2$, and since at most one speed shares the factor of $3$, the other speed has to be co-prime, satisfying the condition that $2v_3 = 12 \geq 6$; $(3,6)$ -- the same logic as the previous case. The other two speeds can't share the factor of $3$, so one of the speeds must be co-prime to $6$, satisfying the condition that $2v_3 = 12 \geq 6$; $(5,6)$ -- the three speeds $v_1$, $v_2$, and $v_4$ would share a common factor of $5$, which is covered by Theorem \ref{the1}.
            \item[--] $v_3 = 5$: This would lead to three speeds sharing a common factor of $5$, which is covered by Theorem \ref{the1}.
            \item[--] $v_3 = 4$: At most one of $v_1$ and $v_2$ can share the factor of $2$, so the other speed must be co-prime to $v_3$, satisfying the condition that $2v_3 = 8 \geq 6$.
            \item[--] $v_3 = 3$: Here, the proof gets a bit tricky. While one of the speeds must be co-prime to $3$, the other speed can be a multiple of $3$ without incurring the already proven cases of $g = 4$ or $g \geq 6$. We first remark that $(v_4,v_3)$ can't be $(2,3)$ as it falls into the proven $v_3 \equiv \pm v_4$ case, so it can only be $(1,3)$ as $v_4 < v_3$. Having $1$ is very crucial, and we reserve the proof in the explanation for $v_3 = 2$ below.
            \item[--] $v_3 = 2$: Since $v_3 > v_4$, the only possible set of speeds is $(v_4,v_3) = (1,2)$. Thus for both $v_3 = 2$ or $3$, we can guarantee $v_4 = 1$. Here, we abandon our original goal of ensuring there are $3$ valid positions in the context of pre-jumping with $h / 5$,
            and instead, we try to ensure there are $3$ valid positions for runner $4$, who has the speed $1$. Direct computation reveals that this is true for 
            $$tv_4 = t \in \left[\frac{1}{20},\frac{3}{20}\right] \bmod {\frac{1}{5}}.$$
            More formally, such a $t$ ensures there are $3$ valid $h$ solutions for
            $$\Norm{\left(\frac{h}{5} + t\right)v_4} \geq \frac{1}{4}$$
            as desired. We remark that the constraint on $t$ can actually be rewritten as
            $$\norm{5t} \geq \frac{1}{4}.$$
            In other words, this is a valid instance of the Lonely Runner Conjecture through the following setup
            $$\exists t \in \mathbb{R},   \min(\Norm{tv_1},\Norm{tv_2},\Norm{5t}) \geq \frac{1}{4} \iff \ML(v_1,v_2,5) \geq \frac{1}{4},$$
            which is true by the Lonely Runner Conjecture with $n=3$. 
            
            \item[--] $v_3 = 1$: By assumption $v_3 > v_4$, so this case cannot happen.
            
        \end{itemize}
        
    \end{itemize}
    At last, we have discussed all of the cases and can conclude that if any two speeds share a common factor $g > 3$, the loneliness value is at least $\frac{1}{4}$.
\end{proof}

\subsection{Almost very lonely runners when $g = 3$}
\label{sec44}
Experimental evidence suggests that a similar result of $\ML \geq \frac{1}{4}$ should be achievable as with the case of the previous section for $g > 3$. However, there is one glaring family of speeds that defy such a conclusion, particularly those of the form
$$(v_1,v_2,v_3,v_4) = (1,2,3,12k)$$
for some integer $k$. One might recognize the pattern within this construction. It utilizes a tight set of speed for $3$ runners $\ML(1,2,3) = \frac{1}{4}$ and forces a lower loneliness value by finding value $v_4$ such that $\Norm{tv_4} = 0$ at the ``local maximum times'', which in this case relies on $v_4 = 12x$. However, this family seems to also be the only exception to the rule. Therefore, in this subsection, we prove the following theorem.

\begin{theorem}
    \label{the4}
    Let $v_1 < v_2 < v_3 < v_4$ be positive integers with $\gcd(v_1,v_2,v_3,v_4) = 1$. Then if any two of the speeds share a common factor $g = 3$,
    $$\ML(v_1,v_2,v_3,v_4) \geq \frac{1}{4}$$
    unless $(v_1,v_2,v_3) = (1,2,3)$ and $v_4 = 12k$ for some natural number $k$.
\end{theorem}

\begin{proof}
    We first forego the assumption that the speeds are strictly increasing. Then without the loss of generality, let $\gcd(v_1,v_2) = g = 3$. If either $v_3$ or $v_4$ share a common factor with $g$, then Theorem \ref{the1} dictates the loneliness value must be at least $\frac{1}{4}$. Hence, we focus on when $v_3$ and $v_4$ are both co-prime to $g$. Because $g = 3$, the following is always true,
    $$v_3 \equiv \pm v_4 \bmod g.$$
    And per usual, we focus on $v_3 \equiv v_4 \bmod g$ since the other case can be achieved by simply flipping signs. Then, the pre-jump of $t' = t + \frac{h}{g}$ would always fix the distance between runners $3$ and $4$. Thus, like with the previous proofs, we need to bound the distance
    $$\norm{t(v_3 - v_4)} \leq \frac{1}{2} - \frac{1}{g} = \frac{1}{6}.$$
    In other words, we need to find a time $t$ such that
    \begin{align*}\min(\norm{tv_1},\norm{tv_2}) \geq \frac{1}{4} \quad \text{and} \quad \norm{t(v_3 - v_4)} \leq \frac{1}{6}.\end{align*}
    Here, we once again have two cases due to the caveat of the distinct speed requirements for the Shifted Lonely Runner problem: (1) $|v_3 - v_4| \not\in \{v_1, v_2\}$ and (2) $|v_3 - v_4| \in \{v_1,v_2\}$.

    \par \textbf{Case 1:} $|v_3 - v_4| \not\in \{v_1,\, v_2\}$: We notice that $v_1$, $v_2$, and $|v_3 - v_4|$ are all multiples of $g$, so for the sake simplicity, we define $a = v_1 / g$, $b = v_2 / g$, and $c = |v_3 - v_4| / g$. There are two sub-cases:
    \begin{itemize}
        \item $c \leq a + b$: We consider runners $1$ and $2$, focusing on times of the form
        $$t = \frac{lu}{a + b}$$
        for some integers $l$ and $u$ such that $au = 1 \pmod {a + b}$. The following proof shares many similarities with Lemma \ref{lem2}, except, in this case the range of values within the valid region for $luc$ is only $\frac{a + b}{3}$ long instead of $\frac{a + b}{2}$. Here we again do casework on the value of $z = uc \pmod {a + b}$. We focus on $z \leq \frac{a + b}{2}$ since the remaining case is symmetric. Furthermore, we define $m = a + b$.
        \begin{itemize}
            \item [--] $z = 0$: In this case, $\|t(v_3 - v_4)\| = 0 \leq \frac{1}{6}$, so it is equivalent to the runners running by themselves.
            \item [--] $z = 1$: This case is impossible since $c \leq a + b$ and $c \neq a$ or $b$, resulting in a contradiction $au \equiv cu \equiv 1 \pmod{a+b}$. Note that if $c > a + b$, this is no longer true.
            \item [--] $2 \leq z \leq \frac{m}{3}$: For this interval, it is impossible for $c$ to skip the valid region of $\frac{m}{3}$ length. Runners $1$ and $2$ cover a range of at least $\frac{\floor{\frac{m}{2}} - 1}{m}$. Since $z \geq 2$, the third runner with speed $|v_3 - v_4|$ or $z$ under modulo $a + b$ is at least twice as fast, so it covers a region of $2 \cdot \frac{\floor{\frac{m}{2}} - 1}{m}$. By the Pigeonhole Principle, if this region is at least $\frac{2}{3}$, since the stride $z$ is at most $\frac{m}{3}$, it must pass through the $\frac{1}{3}$ valid region at some point. Thus we attain the following inequality
            $$2 \cdot \frac{\floor{\frac{m}{2}} - 1}{m} \geq \frac{2}{3}.$$
            Solving gives that all values of $m = a + b \geq 8$ satisfy the inequality.
            \item [--] $\frac{m}{3} < z < \frac{m}{2}$: We denote $L_0$ as the set of values of integer $p$ such that $-\frac{m}{6} \leq p \leq \frac{m}{6} \bmod m$. This is the set of positions that we wish $lz$ would land in. We then consider the set of positions $L_1$ such that one more step of size $z$ would cause the runner to land in $L_0$. More formally, it is
            $$L_1 = L_0 - z \bmod m.$$
            Note that since $z > \frac{m}{3}$, $L_0 \cap L_1 = \emptyset$. Moreover, we have $|L_0| = |L_1| \geq \floor{\frac{m}{3}}$. There are at least $\floor{\frac{m}{2}}$ number of valid $l$ positions such that $\frac{m}{4} \leq (lu)a = l(ua) = l \leq \frac{3m}{4}$, meaning that runners $1$ and $2$ are in the valid region. As long as one of the first $\floor{\frac{m}{2}} - 1$ positions of $(lu)c = l(uc) = lz$ is in $L_0$ or $L_1$, all the desired conditions are satisfies. By the Pigeonhole Principle, we derive the following inequality
            $$\Floor{\frac{m}{2}} - 1 > m - |L_0| - |L_1| \implies \Floor{\frac{m}{2}} - 1 > m - 2 \cdot \Floor{\frac{m}{3}},$$
            which is true for all $m \geq 18$.
            \item [--] $z = \frac{m}{2}$: Either $l = \floor{\frac{g + 1}{2}}$ or $l = \floor{\frac{g + 1}{2}} + 1$ should produce the desired $-\frac{1}{6} \leq lz \leq \frac{1}{6}$ value, since runner $3$ basically oscillates between the $0$ and $\frac{1}{2}$ point.
        \end{itemize}
        Thus combining all three cases, we see that when $a + b \geq c$, the desired conditions are always satisfied as long as $a + b \geq 18$. This bounds two of the speeds to be between $0 \leq a,b < 18 \implies 0 \leq v_1,v_2 < 54$. Without the loss of generality, we assume $v_3 < v_4$, so if $v_3$ is very large, we can apply Lemma \ref{lem4} ``Two Very Fast Runners'' with $L \geq 1/3$ (trivial bound from the $(v_1, v_2)$ two-runner setup) and $v_3 \geq 1/(1/3 - 1/4) \cdot 3 \cdot 54 = 1944$. Otherwise, if $v_3$ is small, but $v_4$ is very large, we can apply Lemma \ref{lem3} ``One Very Fast Runner'' with $v_1, v_2 \leq 54$, $v_3 \leq 1944$, $L = \ML(v_1, v_2, v_3)$ and some corresponding $\epsilon$ for each $L$ (this turns out to be quite a lot of cases to simulate, but we can skip most of them with our previous reductions. In addition, $\ML(v_1, v_2, v_3)$ is generally not super tight). This leaves us with a finite number of cases to check, which we manually verified with a computer algorithm. We note that this method of bounding $2$ speeds and using the ``One/Two Very Fast Runner(s)'' lemmas is quite prevalent in the latter part of the discussion as well. We will henceforth omit the details after bounding $2$ speeds for the sake of minimizing redundancy.
        \item We now consider the case when $c \geq a + b$. We want $\norm{tc} \leq \frac{1}{6}$, so we might as well set it to $0$ by focusing on times of the form $t = \frac{k}{c}$ for some integer $k$. Hence we need to find $k$ such that
        $$\min\left(\Norm{\frac{k}{c}a},\Norm{\frac{k}{c}b}\right) \geq \frac{1}{4}.$$
        It is tempting to apply Lemma \ref{lem2} with $g = c$, but we cannot ensure either $a$ or $b$ is co-prime to $c$. We can resolve this by dividing $c$ by one of $a$ or $b$'s  greatest common divisor. More formally,
        $$\frac{c}{\min(\gcd(a,c),\gcd(b,c))} \geq 6.$$
        We first rewrite the equation as below
        \begin{align*}c = A \cdot \gcd(a,c) \quad \text{and} \quad c = B \cdot \gcd(b,c).\end{align*}
        Observe that $\gcd(a,c)$ is coprime to $\gcd(b,c)$ since $\gcd(a,b) = 1$. This implies $B$ must be a multiple of $\gcd(a,c)$, thus $B \geq \gcd(a,c)$. Hence,
        $$c = A \cdot \gcd(a,c) \leq A \cdot B.$$
        Assuming the conditions for Lemma \ref{lem2} is not met, then $\max(A,B) \leq 5$, so the inequality above further implies
        $$c \leq A \cdot B \leq 5 \cdot 5 = 25.$$
        This again gives us a bound for $a$ and $b$ as $a + b \leq 25$, so we use the same argument as the prior proof with ``Very Fast Runners'' Lemmas for large speeds and manually verify a finite number of small cases, this time with $v_1, v_2 \leq 75$ instead of $54$.
    \end{itemize}

    \par \textbf{Case 2:} $|v_3 - v_4| \in \{v_1,\, v_2\}$. For the sake of simplicity, let $|v_3 - v_4| = v_1$, and $v_3 > v_4$. The central challenge/motivation for this case is that it is impossible to ensure $\norm{t(v_3 - v_4)} = \norm{tv_1} \leq \frac{1}{6}$ and $\norm{tv_1} \geq \frac{1}{4}$ at the same time since they are directly contradictory. Thus, we alter our approach and fix the position of runner $3$ between $[-\frac{1}{12},\frac{1}{12}] \bmod {\frac{1}{3}}$. This is equivalent to $\norm{t(3v_3) + 0.5} \geq \frac{1}{4}$. Under this constraint, there are always $2$ valid positions out of the $3$ from $(t + h / g)v_3$ pre-jump. This allows us to sharpen the required distance between runners $3$ and $4$ to only $\frac{1}{3}$ — no matter how you shift the three points, you have two points in the valid region, and at least one of them is at least $1/3$ away from the invalid region for a fixed direction. We formalize these constructions into the following conditions for some real number $t$,
    \begin{align*}\min(\norm{tv_1},\norm{tv_2},\norm{t(3v_3) + 0.5}) \geq \frac{1}{4} \quad \text{and} \quad \norm{t(v_3 - v_4)} = \norm{tv_1} \leq \frac{1}{3}.\end{align*}
    We focus on time values in the form of $t = \frac{k}{v_1} + p$ for some integer $k$ and real number $p \in [\frac{1}{4v_1},\frac{1}{3v_1}]$. This fixes $\frac{1}{4} \leq \norm{tv_1} \leq \frac{1}{3}$ as the conditions require. We define $a = v_1 / g$, $b = v_2 / g$, and $c = v_3$, transforming the original conditions as follows:
    \begin{align*}\min(\norm{ta},\norm{tb},\norm{tc + 0.5}) \geq \frac{1}{4} \quad \text{and} \quad \norm{ta} \leq \frac{1}{3}.\end{align*}
    Now, we again have three cases that seek to put a bound on $a$.
    \begin{itemize}
        \item [--] $c$ is a multiple of $a$: Note that $v_4 = v_3 - v_1 = c - 3a$, so $v_4$ is also a multiple of $a$. If $a > 2$, then $v_1$, $v_3$, and $v_4$ share a common factor greater than $1$, yielding a loneliness value of at least $\frac{1}{4}$ by theorem \ref{the1}.
        \item [--] $c \neq \pm b \bmod a$: We know that $a$ and $b$ are co-prime by definition, so if $a$ is also at least $6$, then the conditions for Lemma \ref{lem2} with $g = a$ are satisfied, yielding the desired loneliness value of $\frac{1}{4}$.
        \item [--] $c \equiv \pm b \bmod a$: We focus on when $c = b \bmod a$ since the other case is symmetric. Furthermore, we assume $c \neq b$ since if they are equal, we can replace $c = v_3$ with $c = v_4$. This doesn't disrupt the generality of our assumption because $v_4 = v_3 + 3a \equiv v_3 \equiv \pm b \bmod a$. Since the difference between $c$ and $b$ are divisible by $a$ and nonzero, $|c - b| \geq a$. We note back that the starting position of the runner with velocity $a$ can be anywhere between $[\frac{1}{4},\frac{1}{3}]$, yielding a contiguous range of valid starting time $t \in [\frac{1}{4a},\frac{1}{3a}]$ of length $\frac{1}{12a}$. This means that the value of $t(b - c) + 0.5$ changes by at least $\frac{1}{12a} \cdot |b - c| \geq \frac{1}{12a} \cdot a = \frac{1}{12}$. By the Pigeonhole Principle, there must exist time $t$ such that $\norm{t(b - c) + 0.5}$ is at most $\frac{1}{2} - \frac{1}{24}$. Using the results from previous proofs, we know that for times of the form $t + \frac{h}{a}$, there  exists a valid $h$ such that $\min(\Norm{(t + \frac{h}{a})b},\Norm{(t + \frac{h}{a})c + 0.5)}) \geq \frac{1}{4}$ if there exists time $t$ where
        $$\norm{t(b - c) + 0.5} \leq \frac{1}{2} - \frac{1}{a}.$$
        Since we have bounded the norm below $\frac{1}{2} - \frac{1}{24}$, this should be true for all $a \geq 24$.
    \end{itemize}
    Combining the three cases, we learn that if $a \geq 24$, the loneliness value is at least $\frac{1}{4}$. This translates to a bound of $v_1 = 3 \cdot a \geq 72$. However, this approach only bounds one of the speeds, which is insufficient to invoke the ``Very Fast Runners'' Lemmas; thus, we now attempt to bound $c$. Here, we utilize the familiar technique of fixing runner $3$ at $\frac{1}{4}$ by only checking times of the form
    $$t = \frac{k}{c} + \frac{1}{4c}.$$
    This has the effect of expanding the valid region of value for $t(v_3 - v_4) = tv_1$ to $[\frac{1}{4},\frac{1}{2}] \cup [\frac{2}{3},\frac{3}{4}]$. The left interval gets expanded because with runner $3$ fixed at $\frac{1}{4}$, the gap between runners $3$ and $4$ can be as large as $\frac{1}{2}$. However, in the negative direction, our other point in the valid region is $\frac{1}{4} + \frac{1}{3}$, which only has a distance of $\frac{1}{3}$ before reaching the invalid region. To apply an analog of Lemma \ref{lem2}, We ensure that the denominator of $k / c$, $c$, is co-prime to $a$ and $b$ by dividing $c$ by their greatest common divisors. We focus on times of the form
    $$t = \frac{k}{\frac{c}{\gcd(a,c) \cdot \gcd(b,c)}} + \frac{1}{4c}.$$
    Note that $\gcd(a,c)$ and $\gcd(b,c)$ are both at most $2$. Otherwise, if they are $3$, then this would imply all $4$ speeds are multiples of $3$, contradicting the assumption, or if the gcd is at least $4$, then it would fall into the previously proven $g > 3$ cases. Furthermore, $\gcd(a,c)$ is co-prime to $\gcd(b,c)$ since $\gcd(a,b) = 1$, so their product is at most $1 \cdot 2 = 2$. Hence, we get that
    $$\frac{c}{\gcd(a,c) \cdot \gcd(b,c)} \geq \frac{c}{2}.$$
    We denote this new value as $d=\frac{c}{2}$ and claim that if $d$ is sufficiently large, then the loneliness value is at least $\frac{1}{4}$. Let $u$ be the modular inverse of $b$ such that $ub = 1 \bmod d$, then we again need to do casework on the value of $z \equiv ua \bmod d$ to gather a bound on $d$, where $1 \leq z \leq d/2$. We omit the case when $\frac{d}{2} < z < d$ as it is symmetric with trivial modification.
    \begin{itemize}
        \item $z = 1$: In this case,  runners $1$ and $2$ are synchronized. Instead of fixing the runner $3$ at $\frac{1}{4}$, we abandon our previous restriction and instead fix runner $3$ at $0$ with times of the form
        $$t = \frac{k}{d}.$$
        This eliminates the $\frac{1}{4c}$ factor, so $\norm{tv_1} = \norm{tv_2}$ for all $k$. With this new time, the valid range of positions for $tv_1$ is instead $[\frac{1}{4},\frac{1}{3}] \cup [\frac{2}{3},\frac{3}{4}]$. Hence, $k = u \cdot \left(\floor{\frac{d}{4}} + 1\right) / d$ should satisfy both runners $1$ and $2$.
        \item $2 \leq z \leq \frac{d}{4}$: For this interval, it is impossible for runner $1$ to skip over the $[\frac{1}{4},\frac{1}{2}]$ interval, since the stride is less than the interval length. Runner $2$ travels a contiguous region of length at least $\Floor{\frac{d}{2} - 1} / d$. Since runner $1$ is at least twice as fast as runner $2$, it must travel a contiguous region of length at least $2 \cdot \Floor{\frac{d}{2} - 1} / d$, and by the Pigeonhole Principle, if this quantity exceeds $\frac{3}{4}$, it must pass through the $[\frac{1}{4},\frac{1}{2}]$ interval at some point. Thus, we derive the following inequality
        $$2 \cdot \frac{\Floor{\frac{d}{2}} - 1}{d} \geq \frac{3}{4}.$$
        This is true for all $d \geq 10$.
        \item $\frac{d}{4} < z < \frac{d}{2}$ but $z \neq \frac{d}{3}$: We denote $L_0$ as the set of $\frac{h}{d}$ such that $\frac{h}{d} \in [\frac{1}{4},\frac{1}{2}] \cup [\frac{2}{3},\frac{3}{4}]$. This is the set of positions that we wish runner $1$ would land in. We then consider the set of positions $L_1$ such that one more step of size $z$ would cause the runner to land in $L_0$. Since $z > \frac{d}{4}$, the two intervals of $[\frac{1}{4}, \frac{1}{2}]$ are disjoint, so the intersection of the two sets is of length at most $\frac{3}{4} - \frac{2}{3} = \frac{1}{12}$. Hence we have 
        $$|L_0 \cup L_1| \geq 2 \cdot \Floor{\frac{d}{4}} + \Floor{\frac{d}{12}}.$$
        There are at least $\floor{\frac{d}{2}}$ number of valid positions for runner $2$, and if one of the first $\floor{\frac{d}{2}} - 1$ positions is in either $L_0$ or $L_1$, the desired conditions are met. Thus by the Pigeonhole Principle, we derive the following inequality
        $$\Floor{\frac{d}{2}} - 1 > d - |L_0 \cup L_1| \implies \Floor{\frac{d}{2}} - 1 > d - \left(2 \cdot \Floor{\frac{d}{4}} + \Floor{\frac{d}{12}}\right).$$
        This is true for all $d \geq 36$.
        \item $z = \frac{d}{3}$ or $\frac{d}{2}$: These two edge-cases are special because it forces runner $1$ to cycle around the same $2$ or $3$ points, thus defying the conditions needed for the Pigeonhole Principle. However, these cases imply that $b$ and $a$ are both multiples of $\frac{d}{3}$ or $\frac{d}{2}$. So if $d \geq 6$, Theorem \ref{the1} would take care of the rest as it would imply three speeds share a common factor greater than $1$.
    \end{itemize}
    Considering all of the cases, we attain a bound for $d < 36$, which translates to $v_3 < 72$. Furthermore, by assumption, we have $v_3 > v_4$, so this gives us bounds for $2$ speeds. Combining this result with the previous bound of $v_1 < 72$, we have $v_1,v_3,v_4 < 72$. If $v_2$ is sufficiently large, we invoke the ``Very Fast Runner'' Lemma, and otherwise, we have a finite number of cases to manually verify with a computer algorithm. At last, we conclude that if two speeds share a common factor of $3$, and the speeds are not in the form of $(1,2,3,12k)$ for some integer $k$, then the loneliness value is at least $\frac{1}{4}$.
\end{proof}

\subsection{Why is $(1,2,3,12k)$ an exception?}
Upon first glance, the proof from Subsection \ref{sec44} does not directly answer why $(1,2,3,12k)$ is an exception. But a closer inspection reveals the reasoning by recalling Remark \ref{rem1}. If we pass this family of speeds into our proof's casework, we see that it falls into the case when $v_1,v_2,v_3 < 72$. The proof then invokes the ``One Very Fast Runner'' Lemma \ref{lem3} to deal with the unbounded $v_4$. However, a crucial assumption is that $\ML(v_1,v_2,v_3)$ is strictly greater than $\frac{1}{4}$ by a non-trivial amount, producing a finite $v_4 / v_3$ ratio of reasonable magnitude. However, $(v_1,v_2,v_3) = (1,2,3)$ happens to be a (and the only) tight speed set for $n = 3$, resulting in a division by $0$ error. By the proven case of the Loneliness Spectrum Conjecture for $3$ moving runners \ref{the5}, the next smallest loneliness value after $\frac{1}{4}$ is $\frac{2}{7}$, which is indeed non-trivially greater than $\frac{1}{4}$. Finally, any scaling in the form of $(v_1,v_2,v_3) = l \cdot (1,2,3)$ for some integer $l > 1$ would immediately yield a loneliness value of at least $\frac{1}{4}$ due to Theorem \ref{the1}. Hence $(1,2,3,12k)$ remains the only possible family of exceptions to the case of $g = 3$, and thus rightfully not covered.

\begin{prop}[cf. \cite{kravitz} Thm 3.1]
    \label{prop1}
    For any positive integer $k$,
    $$\ML(1, 2, 3, 12k) = \frac{3k}{12k + 1}.$$
\end{prop}

\begin{proof}
    Since the first three runners $(1, 2, 3)$ form a tight speed set, the loneliness value is at most $1/4$. However, this can only happen for times $t = n / 4$ for some integer $0 < n < 4$, which does not satisfy runner $4$. Hence, the loneliness value must be strictly less than $1/4$. By Lemma \ref{lem1}, the relevant speed sums for the denominator are $\{3, 4, 5, 12k + 1, 12k + 2, 12k + 3\}$. We can rule out $3$ and $4$ since they divide $12k$. Out of the remaining denominators, the largest fraction we can form while being less than $1/4$ is
    $$\min\left(\frac{1}{5}, \frac{3k}{12k+1}, \frac{3k}{12k+2}, \frac{3k}{12k+3}\right) = \frac{3k}{12k + 1}$$
    attained at time $t = \frac{3k}{12k + 1}$. Direct computation shows that this indeed yields the desired loneliness value
    $$\min\left(\Norm{\frac{3k}{12k+1}}, \Norm{\frac{6k}{12k + 1}}, \Norm{\frac{9k}{12k + 1}}, \Norm{-\frac{3k}{12k + 1}}\right) = \frac{3k}{12k + 1} $$
    where the fourth term comes from the fact that $12k \cdot \frac{3k}{12k + 1} = (12k + 1)\frac{3k}{12k + 1} - \frac{3k}{12k + 1}$.
\end{proof}

We remark that this family conforms to the Amended Loneliness Spectrum Conjecture. Indeed, it also satisfies the original Spectrum Conjecture, so these examples do not necessitate the proposed amendment. In particular, the loneliness values are of the form $3k/(4(3k) + 1)$. Kravitz also explores this in \cite{kravitz} Theorem 3.1.

\section{Concluding remarks}
\label{sec5}
\subsection{Addressing the discrete spectrum}
\par The proof in Section \ref{sec4} does not directly address the discrete spectrum $[1 / (n + 1),1 / n)$ of the Amended Loneliness Spectrum Conjecture \ref{con2}, but rather focuses on the $\ML(v_1,\ldots,v_n) \geq 1 / n$ cases. While we do not have a strong mathematical argument, our computational experiments revealed only a sparse set of possible values allowed by the form,
$$\ML(v_1, \ldots, v_n) = \frac{s}{sn + k}$$
which we list out in Table \ref{table1}. Although $n=7, 8$ are unconvincing due to the low search bounds.

\begin{table}[h]
\centering
\begin{tabular}{@{}ccc@{}}
\toprule
$n$ & Max value of speeds & Observed values of $k$ \\
\midrule
3 & $1000$ & $\{1\}$ (Proved) \\
4 & $400$  & $\{1, 2\}$ \\
5 & $200$  & $\{1\}$ \\
6 & $100$   & $\{1, 3\}$ \\
7 & $70$   & $\{1, 2\}$ \\
8 & $60$   & $\{1\}$ \\
\bottomrule
\end{tabular}
\caption{Observed $k$ in $\ML=s/(sn+k)$ up to the max values.}
\label{table1}
\end{table}

Previously, we noted that the allowed values of $k$ appears to be highly restricted in our experiments, where $k \in \{1, 2\}$ for $n = 4$, $k \in \{1\}$ for $n = 5$, $k \in \{1, 3\}$ for $n = 6$, and $k \in \{1, 2\}$ for $n = 7$. Indeed, the set of possible $k$ values seem quite sparse and sporadic.

\begin{remark}
    Our computational results would suggest an even sharper variant of this conjecture. Namely, $k \leq n/2$. However, due to the computationally intensive nature of running experiments with high $n$ values, we could not run enough experiments to be confident of this conjecture. Perhaps a more computationally efficient simulation method with $n = 8$ and $n = 9$ would help settle this uncertainty. Even then, the new conjecture would not be able to fully explain why $k=2$ does not appear for $n=6$. Nevertheless, we propose the following stronger conjecture.
\end{remark}

\begin{conj}
    \label{con4}
    For any positive integers $v_1,\ldots,v_n$, we have either
    $$\exists s,k \in \mathbb{N},k \leq n/2,  \ML(v_1,\ldots,v_n) = \frac{s}{ns + k} \text{ or } \ML(v_1,\ldots,v_n) \geq \frac{1}{n}.$$
\end{conj}

\begin{remark}
    It may be tempting to make stronger claims about $k$. Prior to a major revision to the paper in which we extended computational results to $n = 7$ and $8$, we had thought that $k$ always divides $n$, or that $k$ seems to only be $1$ for odd $n$ and $1$ or $n/2$ for even $n$ (which is deceptively elegant). However, when we ran extended experiments on $n = 7$, we discovered numerous counterexamples, such as $\ML(1,2,3,4,5,7,18) = 3/23$ and $\ML(1,3,4,5,7,11,30) = 5/37$.
\end{remark}



\subsection{Future work}
\par Throughout this paper, we have presented a number of questions, ideas, and future directions regarding the (Amended) Loneliness Spectrum Conjecture and its related results. We summarize them below:
\begin{enumerate}
    \item Fully proving the Amended Loneliness Spectrum Conjecture for $n = 4$ following the results of Section \ref{sec4}. Can we complete the proof now that the speeds are almost pairwise co-prime?
    \item Gaining a more refined understanding of the possible values of $k$ by tackling Conjecture \ref{con3} and \ref{con4}. Running more numerical experiments with more efficient simulation methods.
    \item Extending the results of Section \ref{sec3} and more closely study families of speeds in the form $(a,b,b + a,b + 2a,\ldots)$.
\end{enumerate}

\section*{Acknowledgments}

We thank Noah Kravitz for introducing us to the problem and for many helpful discussions. We also thank Francisco Santos for proposing to study the possible values of $k$ in Section \ref{sec5}. Finally, we thank the anonymous reviewers whose details comments greatly improved the exposition of this paper.

\bibliographystyle{abbrv}
\bibliography{references}

\end{document}